\documentclass[reqno, 11pt]{amsart}
\usepackage{mathtools}
\usepackage{amsmath}
\usepackage{amssymb}
\usepackage{yhmath}
\usepackage{graphicx}
\usepackage{ mathrsfs }
\usepackage{bbm}
\usepackage{xcolor}
\usepackage{tikz-cd}
\usepackage{tikz}
\usetikzlibrary{patterns}
\usepackage{hyperref}

\DeclareMathAlphabet{\mathpzc}{OT1}{pzc}{m}{it}

\usepackage{thmtools}
\usepackage{thm-restate}

\usepackage{caption}

\newtheorem{theorem}{Theorem}[section]

\newtheorem*{claim*}{Claim}

\newtheorem{lemma}[theorem]{Lemma}
\newtheorem{lem}[theorem]{Lemma}
\newtheorem{corollary}[theorem]{Corollary}

\newtheorem{cor}[theorem]{Corollary}

\theoremstyle{definition}
\newtheorem{Def}[theorem]{Definition}

\theoremstyle{remark}

\newtheorem{rmk}[theorem]{Remark}

\numberwithin{equation}{section}

\newcommand{\op}{\operatorname}

\newcommand{\be}{\begin{equation}}
\newcommand{\ee}{\end{equation}}
\newcommand{\Ga}{\Gamma}
\newcommand{\bc}{\mathbb C}

\renewcommand{\H}{\mathbb H}

\newcommand{\N}{\mathbb N}
\newcommand{\ga}{\gamma}

\newcommand{\La}{\Lambda}
\newcommand{\inte}{\op{int}}
\newcommand{\ba}{\backslash}

\newcommand{\cal}{\mathcal}
\newcommand{\br}{\mathbb R}

\newcommand{\Isom}{\op{Isom}}
\newcommand{\PSL}{\op{PSL}}
\newcommand{\F}{\cal F}

\newcommand{\bH}{\mathbb H}

\newcommand{\diam}{\op{Diam}}

\newcommand{\G}{\Gamma}

\newcommand{\e}{\varepsilon}

\renewcommand{\L}{\mathcal L}

\renewcommand{\S}{\mathbb S}

\newcommand{\Gr}{\Gamma_\rho}

\renewcommand{\c}{\bc}

\newcommand{\id}{\op{id}}
\newcommand{\dGa}{\mathfrak R_{\op{disc}}(\Ga)}
\newcommand{\Mob}{\text{M\"ob}}
\newcommand{\PGL}{\op{PGL}}

\begin{document}

\title{Rigidity of Kleinian groups via self-joinings}

\author{Dongryul M. Kim}
\address{Department of Mathematics, Yale University, New Haven, CT 06511, USA}
\email{dongryul.kim@yale.edu}

\author{Hee Oh}
\address{Department of Mathematics, Yale University, New Haven, CT 06511, USA and Korea Institute for Advanced Study, Seoul, South Korea}
\email{hee.oh@yale.edu}

\thanks{
 Oh is partially supported by the NSF grant No. DMS-1900101.}
 
\begin{abstract} 
Let $\Gamma<\op{PSL}_2(\mathbb{C})\simeq \op{Isom}^+(\mathbb{H}^3)$ be a finitely generated non-Fuchsian Kleinian group whose ordinary set $\Omega=\S^2-\Lambda$ has at least two components. Let $\rho : \Gamma \to \op{PSL}_2(\mathbb{C})$ be a faithful discrete non-Fuchsian representation with boundary map $f:\Lambda\to \mathbb{S}^2$ on the limit set.

In this paper, we obtain a new rigidity theorem:
 if $f$ is {\it conformal on $\La$}, in the sense that $f$ maps every circular slice of $\Lambda$ into a circle, then $f$ extends to a M\"obius transformation $g$ on $\S^2$ and 
$\rho$ is the conjugation by $g$.
Moreover, unless $\rho$ is a conjugation, 
the set of circles $C$
such that $f(C\cap \Lambda)$ is contained in a circle has empty interior in the space of all circles meeting $\Lambda$.
This answers a question asked by McMullen on the rigidity of maps $\Lambda\to \S^2$ sending vertices of every tetrahedron of zero-volume to vertices of a tetrahedron of zero-volume.

The novelty of our proof is a new viewpoint of relating the rigidity of $\Gamma$ with the higher rank dynamics of the self-joining $(\op{id} \times \rho)(\Gamma)<\op{PSL}_2(\mathbb{C})\times \op{PSL}_2(\mathbb{C})$.
\end{abstract}

\maketitle
\section{Introduction}

Let $\Gamma<\op{PSL}_2(\mathbb{C}) = \Isom^+(\H^3)$ be a finitely generated torsion-free Kleinian group. 
Consider the following discreteness locus of $\Ga$ in the space of representations of $\Gamma$ into $\PSL_2(\c)$:
$$\mathfrak R_{\op{disc}}(\Ga)=\{
\rho:\Ga \to \PSL_2(\bc): \text{discrete, faithful}\};$$
each $\rho\in \mathfrak R_{\op{disc}}(\Ga)$ gives  rise to a hyperbolic manifold
 $\rho(\Ga)\ba \bH^3$ which is homotopy equivalent to $\Ga\ba \bH^3$. 
Another commonly used notation for $\mathfrak R_{\op{disc}}(\Ga) $
is $\cal{AH}(\Gamma)$ where $\cal H$ stands for hyperbolic and $\cal A$ for the topology on
this space given by the algebraic convergence (cf. \cite{Thurston1986hyperbolic}).

We denote by $\Mob(\S^2)$ the group of all M\"obius transformations on $\S^2$, by which we mean the group generated by inversions with respect to circles in $\S^2$. As well-known,
$\Mob(\S^2)$ is equal to the group of conformal automorphisms of $\S^2$. The group $\PSL_2(\c)$
can be identified with the subgroup consisting of compositions of even number of inversions with respect to circles in $\S^2$; in particular,
it is a normal subgroup of $\Mob(\S^2)$ of index two. This means that
conjugations by elements of $ \Mob(\S^2)$ are contained in
$\mathfrak R_{\op{disc}}(\Ga)$; we call them {\it trivial} elements of $\dGa$. Note that $\rho\in \mathfrak R_{\op{disc}}(\Ga)$ is trivial if and only if $\Ga\ba \bH^3$ and $\rho(\Ga)\ba \bH^3$ are isometric to each other.

The rigidity question on $\Gamma$ concerns a criterion on
when a given representation $$\rho\in \dGa$$ is trivial. 
Denote by $\La\subset \S^2$ the limit set of $\Ga$, that is, the set of all accumulation points of $\Ga (o)$, $o\in \bH^3$.
 A $\rho$-equivariant  continuous embedding 
$$f : \La \to \S^2$$  is called a $\rho$-boundary map. 
 There can be at most one $\rho$-boundary map. Two important class of representations admitting boundary maps are as follows. Firstly,
if both $\Ga$ and $\rho(\Ga)$ are geometrically finite,
and $\rho$ is type-preserving, then the $\rho$-boundary map always exists by Tukia \cite{Tukia1985isomorphisms}. Secondly,
if $\rho$ is a quasiconformal deformation of $\Ga$, i.e., there exists a quasiconformal homeomorphism $F:\S^2\to \S^2$ such that for all $\ga\in \Ga$,
$\rho(\ga)= F\circ \ga \circ F^{-1}$, 
then  the restriction of $F$ to $\La$ is the $\rho$-boundary map.

The fundamental role played by the boundary map in the study of rigidity of $\Ga$ is well-understood,  going back to the proofs of Mostow's and Sullivan's rigidity theorems (\cite{Mostow1968quasiconformal}, \cite{Mostowbook}, \cite{Sullivan1981ergodic}). 
By the Ahlfors measure conjecture (\cite{Ahlfors1964finitely}, \cite{Ahlfors1966fundamental}) now confirmed by
the works of Canary \cite{Canary1993ends}, Agol \cite{agol2004tameness} and Calegari-Gabai \cite{Calegari2006shrinkwrapping}, the limit set $\La$ is either all of $\S^2$ or of Lebesgue measure zero.  Mostow rigidity theorem  (\cite{Mostow1968quasiconformal}, \cite{Mostowbook}, \cite{Prasad1973}) says  that if $\Ga$ is a lattice, that is, if $\Ga\ba \bH^3$ has finite volume,
then any $\rho\in \mathfrak R_{\op{disc}}(\Ga)$ is trivial; he obtained this by showing that the $\rho$-boundary map has to be conformal on $\S^2$.
More generally, for any finitely generated Kleinian group $\Ga$ with $\La=\S^2$,  Sullivan showed that any quasiconformal deformation of $\Ga$ is trivial \cite{Sullivan1981ergodic}. In fact, Sullivan's original theorem says that
any $\rho$-equivariant quasiconformal homeomorphism of $\S^2$ which is conformal on the ordinary set $\Omega=\S^2-\La$ is a M\"obius transformation. However Ahlfors measure conjecture implies that this is meaningful only when $\La=\S^2$ (cf. \cite[Section 3.13]{Marden2016hyperbolic}).

In this paper, we concern the case when $\La\ne \S^2$. For example, any geometrically finite Kleinian group which is not a lattice satisfies $\La\ne \S^2$ \cite{Sullivan1984entropy}. We prove that if the $\rho$-boundary map is {\it{conformal on $\La$}}, then $\rho$ is trivial, provided  the ordinary set $\Omega=\S^2-\La$ has at least two connected components. 
By the ``conformality of $f$ on $\La$'', we mean that $f$ maps {\it circles in $\La$} into circles. 
\subsection*{Circular slices}

The main result of this paper is the following rigidity theorem in terms of the behavior of $f$ on circular slices of $\La$: a circular slice of $\La$ is a subset of the form $C\cap \La$ for some circle
$C\subset \S^2$. We denote by $\cal C_\La$ the space of all circles in $\S^2$ meeting $\La$.

\begin{theorem}\label{m1} Let $\Ga<\PSL_2(\c)$ be a finitely generated Zariski dense Kleinian group whose ordinary set $\Omega$ has at least two components.  Let $\rho\in \dGa$ be a Zariski dense representation with boundary map $f:\La\to \S^2$. 

If $f$ maps every circular slice of $\La$ into a circle,
then $\rho$ is a conjugation by some $g\in \op{M\ddot{o}b}(\S^2)$ and $f=g|_\La$. 

Moreover, unless $\rho$ is a conjugation, the following subset of $\cal C_\La$
\be\label{ccc} \{C\in \cal C_\La: f(C\cap \La)\text{ is contained in a circle}\}\ee
has empty interior.
\end{theorem}

We call $\La$ {\it doubly stable} if for any $\xi\in \La$,
there exists a circle $C\ni \xi$ such that
for any sequence of circles $C_i$ converging to $C$, $\#\limsup (C_i\cap \La)\ge 2$.
The assumption that $\Ga$ is finitely generated with $\Omega$ disconnected was used only to ensure that
 $\La$ is doubly stable (Lemma \ref{stable}, Theorem \ref{m3}). 
\begin{rmk} \label{rmk.geomfin}
\begin{enumerate}
   \item This theorem holds rather trivially when $\Lambda=\S^2$, in which case all circular slices of $\La$ are  circles.

\item If $\Ga<\PSL_2(\bc)$ is geometrically finite  with connected limit set, then $\Omega$ is disconnected (cf. \cite[Chapter IX]{Maskit1988Kleinian}); hence Theorem \ref{m1} applies.

\end{enumerate}
\end{rmk}

\subsection*{Tetrahedra of zero-volume}

A quadruple of points in $\S^2$ determines an ideal tetrahedron of the hyperbolic 3-space $\mathbb H^3$. Gromov-Thurston's proof of Mostow rigidity theorem for closed hyperbolic $3$-manifolds uses the fact  that a homeomorphism of $\S^2$ mapping vertices of a maximal volume tetrahedron to vertices of a maximal volume tetrahedron is a M\"obius transformation (\cite{Gromov1981hyperbolic} \cite[Chapter 6]{thurston2022geometry}). In view of this, Curtis McMullen asked us whether one can consider the other extreme type of tetrahedra, namely, those of zero-volume in the study of rigidity of~$\Ga$.  

Noting that $f:\La\to \S^2$ maps every circular slice of $\La$ into a circle if
and only if $f$ maps any quadruple of points in $\La$ lying in a circle
into a circle,
the following is a reformulation of Theorem \ref{m1}, which answers McMullen's question in the affirmative:

\begin{theorem}\label{main2}
Let $\Ga, \rho$ be as in Theorem \ref{m1}.
If the $\rho$-boundary map $f:\La\to \S^2$ maps vertices of every tetrahedron of zero-volume to vertices of  a tetrahedron of zero-volume, then
$f$ is the restriction of a M\"obius transformation $g$ and $\rho$ is the conjugation by $g$.
\end{theorem}

\subsection*{Cross ratios} Theorem \ref{main2} can also be stated in terms of cross ratios: 
note that for four distinct points $\xi_1, \xi_2, \xi_3, \xi_4\in \hat \c$,
the cross ratio $[\xi_1: \xi_2: \xi_3 : \xi_4]$
is a real number if and only if
all $\xi_1, \xi_2, \xi_3, \xi_4$ lie in a circle.

\begin{corollary}
Let $\Ga, f$ be as in Theorem \ref{m1}.
 If
$[f(\xi_1) : f(\xi_2): f(\xi_3): f( \xi_4)]\in \br$ for any distinct
$\xi_1, \xi_2, \xi_3, \xi_4\in \La$ with $[\xi_1: \xi_2: \xi_3: \xi_4]\in \br$, 
then $f$ extends to a M\"obius transformation on $\hat\c$.  
\end{corollary}

\subsection*{On the proof of Theorem \ref{m1}} 
The novelty of our approach is to relate the rigidity question for a Kleinian group $\Ga<\PSL_2(\c)$ with the dynamics of
one parameter diagonal subgroups on the quotient of a higher rank semisimple {\it real} algebraic group $G=\PSL_2(\c)\times \PSL_2(\c)$ by a self-joining discrete subgroup.

For a given $\rho\in \dGa$, we consider the following self-joining  of $\Ga$ via $\rho$: 
$$\Ga_\rho=(\text{id}\times \rho)(\Ga)=\{(\ga, \rho(\ga)): \ga \in \Ga\},$$ 
which is a discrete subgroup of $G$. A basic but crucial observation is that $\rho$ is trivial if and only if $\Ga_\rho$
is not Zariski dense in $G$ (Lemma \ref{Zdense}).
Our strategy is then to
 prove that if $f$ maps {\it too many} circular slices of $\La$ into circles, then $\Ga_\rho$ cannot be Zariski dense in $G$. We achieve this by considering the action of $\Ga_\rho$ on the space $\cal T_\rho$ of all tori in the Furstenberg boundary $\S^2\times \S^2$ intersecting the limit set $\La_\rho= \{(\xi, f(\xi))\in \S^2\times \S^2: \xi\in \La\}$. Here a torus means   an ordered pair of circles in $\S^2$.

\begin{enumerate}
    \item On one hand, using the Koebe-Maskit theorem (\cite{Maskit1974Intersection}, \cite{Sasaki1981KoebeMaskit}, see Theorem \ref{koebe})
    and the hypothesis that the ordinary set $\Omega$ has at least $2$ components, we show the existence of a torus $T\in \cal T_\rho$
    such that
    $$T\notin \overline{\Ga_\rho T_0}$$
for any torus $T_0=(C_0, D_0)$ with $f(C_0\cap \La)\subset D_0$; in particular $\overline{\Ga_\rho T_0}\ne \cal T_\rho$.
    
    \item On the other hand, we prove in Theorem \ref{densec} that the Zariski density of $\Ga_\rho$ implies the existence of a dense subset $\tilde \La_\rho$ of $\La_\rho$ such that   $\overline{\Ga_\rho T_0}= \cal T_\rho$
    for
    any torus $T_0$ meeting $\tilde \La_\rho$. Denoting by $A$ the two dimensional diagonal subgroup of $G$, the main ingredients for this step are  the existence of a dense orbit  of some {\it regular} one-parameter diagonal semigroup in the non-wandering set of the $A$-action on $\Ga_\rho\ba G$
    (Theorem \ref{dense}) as well as  a theorem of Prasad-Rapinchuk \cite{Prasad2003existence} on the existence of $\br$-regular elements (Theorem \ref{PR}). Therefore, if the subset \eqref{ccc} has non-empty interior,
   we can find a torus $T_0=(C_0, D_0)$ satisfying that $f(C_0\cap \La)\subset D_0$ and $\overline{\Ga_\rho T_0}=\cal T_\rho$.
\end{enumerate}

The incompatibility of (1) and (2) implies that either
 the subset \eqref{ccc} has empty interior or $\Ga_\rho$ is not Zariski dense in $G$, as desired.

\subsection*{Question} There are several different proofs of Mostow rigidity theorem (\cite{Mostow1968quasiconformal}, \cite{Mostowbook}, \cite{Prasad1973}). By
the viewpoint suggested in this paper, it will be interesting to find yet another proof, which directly shows
 the following reformulation: for any lattice $\G<\PSL_2(\c)$ and $\rho\in \dGa$, the self-joining $\Ga_\rho$ is not Zariski dense in $\PSL_2(\c)\times \PSL_2(\c)$.

\subsection*{Acknowledgements}
We would like to thank Curt McMullen for asking the question formulated as Theorem \ref{main2} as well as for useful comments on the preliminary version.
We would also like to thank Yair Minsky for useful conversations.

\section{Dense orbits in the space of Tori}

  Let $G=\PSL_2(\c)\times \PSL_2(\c)$ and let
  $X=\bH^3\times \bH^3$ be the Riemannian product of two hyperbolic $3$-spaces.  It follows from $\PSL_2(\c)\simeq \op{Isom}^+(\bH^3)$ that $G\simeq \op{Isom}^{\circ}(X)$.
  In the whole paper, we regard $G$ as a {\it real} algebraic group and the Zariski density of a subset of $G$ is to be understood accordingly.
  The action of $\PSL_2(\c)$ on $\bH^3$ extends continuously to the compactification
 $\bH^3\cup \partial \bH^3$ and  its action on $\partial\bH^3\simeq\S^2$ is given by the M\"obius transformation action of $\PSL_2(\bc)$ on $\S^2$.
 We set $\cal F=\S^2\times \S^2$, which coincides with the so-called Furstenberg boundary of $G$.  Note that $\cal F$ is not the geometric boundary of $X$.  Clearly, the action of $G$ extends continuously to the compact space $X\cup \cal F$.
 
 For a Zariski dense subgroup $\Delta$ of $G$,
 its limit set $\Lambda_\Delta\subset \cal F$
 is defined as all possible accumulation points of $\Delta (o)$, $o\in X$,
 {\it on $\cal F$}.
 It is a {\it non-empty} $\Delta$-minimal subset of $\cal F$ 
 (\cite[Section 3.6]{Benoist1997proprietes}, \cite[Lemma 2.13]{lee2020invariant}).

By a torus $T$, we mean an ordered pair  $T=(C_1, C_2)\subset \cal F$ of circles in $\S^2$. The group $G$ acts on the space of tori by extending the action of $\PSL_2(\bc)$ on the space of circles componentwise.
The main goal of this section is to prove the following:
denote by ${\cal T}_{\Delta}$ 
the space of all tori in $\F$ intersecting $ \La_\Delta$.

\begin{theorem} \label{densec} Let $\Delta$ be a 
Zariski dense subgroup of $G$. There exists a dense subset $\tilde{\La}_{\Delta}$ of $\La_{\Delta}$ such that for any torus $T$ with $T\cap \tilde \La_{\Delta} \ne \emptyset$,
the orbit ${\Delta T}$ is dense in  ${\cal T}_{\Delta}$.
\end{theorem}
This theorem may be viewed as a higher rank analogue of \cite[Theorem 4.1]{McMullen2017geodesic}.
The rest of this section is devoted to its proof. 
It is convenient to use the upper half-space model of $\bH^3$ so that
$\partial \bH^3=\c\cup\{\infty\}$.
The visual maps $G\to \cal F$, $g\mapsto g^{\pm}$, are defined as follows:
for $g=(g_1, g_2)\in G$ with $g_i\in \PSL_2(\c)$,
$$g^+=(g_1(\infty), g_2(\infty)) \;\; \text{ and } \;\; g^-=(g_1(0), g_2(0)) .$$

For $t\in \bc$,
we set $a_t=\text{diag}( e^{t/2}, e^{-t/2})$ and define the following subgroups of $G$:
$$ A =\{(a_{t_1}, a_{t_2}) : t_1, t_2\in \br \}\text{ and }
 M =\{(a_{t_1}, a_{t_2}): t_1, t_2\in i \br\} . $$

For  $u=(u_1, u_2)\in \br^2$, we write $a_u=(a_{u_1}, a_{u_2})$
and consider the following one-parameter semisubgroup $$A_u^+=\{a_{tu}: t\ge 0\}.$$

A loxodromic element $h\in \PSL_2(\bc)$
 is of the form
$h=\varphi a_{t_h} m_h \varphi^{-1}$ where $t_h>0$ and $m_h\in \op{PSO}(2)$
are uniquely determined and $\varphi\in \PSL_2(\bc)$. We call $t_h>0$ the Jordan projection of $h$ and $m_h$ the rotational component of $h$.
The attracting and repelling fixed points of $h$ on $\S^2$ are given by $y_h=\varphi (\infty) $ and $y_{h^{-1}}=\varphi (0)$, respectively.

For a loxodromic element $g=(g_1, g_2)\in G$, that is, each $g_i$ is loxodromic, its Jordan projection $\lambda(g)$ and the rotational component $\tau(g)$
are defined componentwise:
$\lambda(g)=(t_{g_1}, t_{g_2}) \in \br^2_{>0}$ 
and $\tau(g)=(m_{g_1}, m_{g_2}) \in M$.

\subsection*{Dense $A_u^+$-orbit}
For a Zariski dense subgroup $\Delta$ of $G$,
we consider the following $AM$-invariant subset $$\cal R_\Delta=\{[g]\in \Delta\ba G: g^+, g^-\in \Lambda_{\Delta}\}.$$

Let $\L=\L_\Delta\subset \br_{\ge 0}^2$ denote the limit cone of $\Delta$, which is the smallest closed cone containing the Jordan projection $\lambda(\Delta)=\{\lambda(\delta):\delta\in \Delta\}$.  The Zariski density of $\Delta$ implies that $\L$ has non-empty interior \cite[Section 1.2]{Benoist1997proprietes}.

We use the following theorem which is an immediate consequence of the result of Dang \cite{dang2020topological} (this also follows from \cite{chow2021local} and \cite{burger2021hopf}):
\begin{theorem} \label{dense} For any Zariski dense subgroup $\Delta<G$
and any $u\in \inte \L_\Delta$,
there exists a dense $A_u^+$-orbit in $\cal R_\Delta$.
\end{theorem}
\begin{proof}
As shown in \cite[Theorem 7.1 and its proof]{dang2020topological},
the semigroup $S^+:=\{a_u^n: n\in \mathbb N\cup \{0\}\}$ acts on $\cal R_\Delta$ topologically transitively: for any non-empty open subsets $\cal O_1, \cal O_2$ of $\cal R_\Delta$,  $\cal O_1 a_u^n \cap \cal O_2\ne\emptyset$
for some $n\in \mathbb N$.
This implies the existence of a dense $S^+$-orbit on $\cal R_\Delta$ (cf. \cite[Proposition 1.1]{Silverman1992maps}). Since $S^+\subset A_u^+$,
this proves the claim.
\end{proof}

 In the following, we fix $u\in \inte\L_{\Delta}$ and
 a dense $A_u^+$-orbit, say $[g_0]A_u^+$, in $\cal R_{\Delta}$, provided by Theorem \ref{dense}. Set 
\be\label{tilde} \tilde \La_\Delta =\Delta g_0^+=\left\{\delta  g_0^+\in \La_{\Delta} : \delta\in \Delta\right\};\ee 
note that this is a dense subset of $ \La_{\Delta}$, as $\La_{\Delta}$ is a $\Delta$-minimal subset.

Denote by $\cal T_{\Delta}^{\spadesuit}$ the space of all tori $T$ with $\# T \cap \La_{\Delta} \ge 2$.

\begin{cor} \label{densetori}
For any torus $T$ meeting $\tilde \La_{\Delta}$,  the closure of
$\Delta T$ contains $\cal T_{\Delta}^{\spadesuit}.$
\end{cor}
\begin{proof} 
Note that $H=\PGL_2(\br)
\times \PGL_2(\br)$ is a subgroup of $G$, as $\PSL_2(\c)=\PGL_2(\c)$.
The space $\cal T$ of all tori in $\cal F$ can be identified with the quotient space $G/H$.
Let $T$ be a torus containing $\delta_0 g_0^+ \in \tilde \La_\Delta$ for some $\delta_0\in \Delta$. By the identification of $\cal T = G / H,$
we may write $T = gH$ for some $g \in G$.
Then
for some $h\in H$, $(gh)^+=\delta_0 g_0^+$.  If we denote by $P$ the stabilizer subgroup of $(\infty, \infty)$ in $G$, which is equal to the product of two upper triangular subgroups of $\PSL_2(\c)$, this implies that for some $p\in P$, $gh = \delta_0 g_0 p $. Write $p=nam$ where $n$ belongs to the strict upper triangular subgroup, $a\in A$ and $m\in M$.
We  claim that
$\overline{[g]hA_u^+} \supset (\cal R_{\Delta}-[g_0] A_u^+) ma$.
Let $x\in \cal R_{\Delta}-[g_0] A_u^+$. Since $\overline{[g_0]A_u^+} = \cal R_{\Delta}$, there exists a sequence $t_i\to +\infty$ such that
$x=\lim_{i\to \infty} [g_0] a_{t_i u} $.
Since $u=(u_1, u_2)\in \inte \L_{\Delta}$,
we have $u_1>0, u_2>0$, and hence
$a_{-t_i u} n a_{t_i u}\to e$ as $i\to \infty$.

Therefore
$$\lim_{i\to \infty} [g] h a_{t_i u}= \lim_{i\to \infty} [g_0] nam a_{t_iu}=
\lim_{i\to \infty} [g_0] a_{t_iu} (a_{-t_i u} n a_{t_i u})  am
=xam ;$$
so $xam\in \overline{[g]hA_u^+} $.
This proves the claim. Since
$\cal R_{\Delta}$ is $AM$-invariant, and $\cal R_{\Delta}-[g_0]AM$ is dense in $\cal R_{\Delta}$ (as $\Lambda_{\Delta}$ is a perfect set),
it follows that 
$$\overline{[g]hA_u^+} \supset \cal R_{\Delta}.$$

Since $A_u^+\subset H$, this implies that $\overline{[g]H}\supset \cal R_\Delta H.$
Since
$\cal R_\Delta H= \Delta \backslash \cal T_{\Delta}^{\spadesuit}$ and $T=gH$, we get
$\overline{\Delta T}\supset \cal T_{\Delta}^{\spadesuit}$, as desired.
\end{proof}

\subsection*{Loxodromic element $\delta\in \Delta$ with $\tau(\delta)$ generating $M$}
We use the following special case of a theorem of Prasad and Rapinchuk \cite{Prasad2003existence}:
\begin{theorem}{\cite[Theorem 1, Remark 1]{Prasad2003existence}} \label{PR}
Any Zariski dense
subgroup $\Delta<G$ contains a loxodromic element $\delta$ such that $\tau(\delta)$ generates a dense subgroup of $M$.
 \end{theorem}

\begin{cor} \label{cor_dense}
If $\Delta$ is Zariski dense in $G$, then $\cal T_{\Delta}^{\spadesuit}$ is dense in ${\cal T}_{\Delta}$.
\end{cor}

\begin{proof} Let $\delta =(\delta_1, \delta_2)\in \Delta$ be as given by
 Theorem \ref{PR}. Since $M$ has no isolated point, there exists a sequence $m_j$, which we may assume tends to $+\infty$, by replacing $\delta$ by $\delta^{-1}$ if necessary,  that
 $\tau(\delta)^{m_j}$ converges to $e$.  It follows that the semigroup generated by $\tau(\delta)$ is also dense in $M$.
Let $T = (C_1, C_2)\in \cal T_\Delta$ be any torus. It suffices to construct a sequence $T_n=(C_{1, n}, C_{2, n}) \in \cal T^{\spadesuit}_\Delta$ which converges to $T$.
We begin by fixing a point $\xi=(\xi_1, \xi_2) \in T \cap \La_{\Delta}$. 
Since $\Delta$ acts minimally on $\La_{\Delta}$, there exists a sequence $\delta_n
=(\delta_{1, n}, \delta_{2, n}) \in \Delta$ such that  that $\delta_n y_{\delta} $ converges to
$\xi$ as $n \to \infty$; recall that $y_\delta\in \cal F$ denotes the attracting fixed point of $\delta$.
Fix a point $\eta =(\eta_1, \eta_2) \in \La_{\Delta}-\{y_\delta, y_{\delta^{-1}}\}$.  

For each fixed $n \in \N$,  note that, as $k \to \infty$, the sequence
$\delta_{n} \delta^k \eta $ converges to $\delta_{n} y_{\delta}$, while rotating around $\delta_n y_{\delta}$ by the amount given by $ \tau(\delta)^k$.
Since $\tau(\delta)$ generates a dense semigroup of $M$,
we can find a sequence $k_n\to \infty $ such that for each $i=1,2$,
$$d(\delta_{i, n} y_{\delta_{i}}, \delta_{i, n} \delta_{i}^{k_n} \eta_i ) < \tfrac{1}{n} \quad \mbox{and} \quad \tfrac{\pi}{2} - \tfrac{1}{n} < \theta_{i, n} < \tfrac{\pi}{2} + \tfrac{1}{n}$$
where $\theta_{i, n}$ is the angle at $\delta_{i, n}y_{\delta_i}$ of the triangle determined by 
the center of $C_i$, $ \delta_{i, n}y_{\delta_i}$ and $\delta_{i, n}\delta_i^{k_n}\eta_i$. 
For each $i = 1, 2$, we now choose $p_i \in C_i - \bigcup_n \{\delta_{i, n}y_{\delta_i}, \delta_{i, n}\delta_i^{k_n}\eta_i\}$ and set $C_{i, n}$ to be the circle passing through $\delta_{i, n}y_{\delta_i}, \delta_{i, n}\delta_i^{k_n}\eta_i$ and $p_i$. 

From the construction, each sequence $C_{i, n}$ converges to the circle tangent to $C_i$ at $\xi_i$ and passing through $p_i\in C_i$, which must be equal to $C_i$ itself; therefore  if we
 set $T_n=(C_{1, n}, C_{2, n})$,
$$T_n \to T \quad \text{as} \quad n\to \infty.$$
Since $T_n\cap \La_{\Delta}$ contains both
$\delta_n y_{\delta}$ and $\delta_n \delta^{k_n} \eta$,  we have $T_n\in \cal T_{\Delta}^{\spadesuit}$. This completes the proof.
\end{proof}

\subsection*{Proof of
Theorem \ref{densec}} It suffices to consider the set $\tilde \La_\Delta$ defined in 
\eqref{tilde} by Corollary \ref{densetori} and Corollary \ref{cor_dense}.

\section{Limits of circular slices and Koebe-Maskit theorem}
Let $\Ga<\PSL_2(\bc)$ be a non-elementary Kleinian group and $\Omega=\S^2-\La$ its ordinary set, i.e., $\La\subset \S^2$ denotes the limit set of $\Ga$. We refer to \cite{Marden2016hyperbolic} and \cite{Matsuzaki1998hyperbolic}  for general facts on the theory of Kleinian groups.

\begin{Def} \begin{enumerate}
    \item We call a circle $C$ doubly stable for $\La$ if 
for any sequence of circles $C_i$ converging to $C$, $\#\limsup (C_i\cap \La)\ge 2$.
\item We call $\La$ doubly stable if for any $\xi\in \La$,
there exists a circle $C\ni \xi$, which is doubly stable for $\La$. \end{enumerate}
\end{Def}

The main goal of this section is to prove the following lemma:
\begin{lem} \label{stable} If $\Ga$ is finitely generated and $\Omega$ is not connected, then $\La$ is doubly stable.
\end{lem}

In the rest of this section, we assume $\Ga$ is finitely generated.
Lemma \ref{stable} is an immediate consequence of the following lemma, since, if $\xi_1, \xi_2\in \Omega$ belong to different components of $\Omega$,
then for any $\xi\in \La$, the circle $C$ passing through $\xi, \xi_1,\xi_2$  is not contained in the closure of any component of $\Omega$.

\begin{lemma}\label{limsuplim} Let $C\subset \S^2$ be a circle such that $C \not\subset \overline{\Omega_0}$ for any component $\Omega_0$ of $\Omega$.
  If $C_n$ is a sequence of circles converging to $C$,
  then\footnote{
For a sequence of subsets $S_n$ in a topological space, we define
$\limsup S_n=\bigcap_{n} \overline{\bigcup_{i\ge n}  S_{i}} .$     } $$ \# \limsup (C_n \cap \Lambda)  \ge 2.$$
\end{lemma}

The main ingredient is the following formulation of 
the Koebe-Maskit theorem 
(\cite[Theorem 6]{Maskit1974Intersection}, \cite[Theorem 1]{Sasaki1981KoebeMaskit}):
\begin{theorem} \label{koebe}
Let $\{\Omega_i\}$ be the collection of all components of the ordinary set $\Omega$.
Then for any $\alpha>2$,
$\sum_{i}\diam ( \Omega_i)^\alpha <\infty $
where $\diam(\Omega_i)$ is the diameter of $\Omega_i$ in the spherical metric on $\S^2$.
\end{theorem}

We will only need the following
immediate corollary of Theorem \ref{koebe}:
\begin{cor} \label{diam} 
 For any $\e > 0$, there are only finitely many components of the ordinary set of $\Ga$ with diameter bigger than $\e$.
\end{cor}

\subsection*{Proof of Lemma \ref{limsuplim}} Given Corollary \ref{diam},
the proof is similar to the proof of \cite[Lemma 8.1]{lee2019orbit}, which deals with the case when all components of $\Omega$ are round disks.

Let $C$ and $C_n \to C$ be as in the statement of the lemma.
It suffices to show that there exists $\e_0 > 0$ such that
$C_{n_i} \cap \Lambda$ contains two points of  distance at least $\e_0$ for some infinite sequence $n_i\to \infty$.
  Suppose not. Then, letting $I_n$ be the minimal connected subset of $C_n$ containing $C_n\cap \La$,
 we have $\diam (I_n)\to 0$ as $n\to \infty$. 

  Setting $\eta= \diam(C)/2$, we have $\diam(C_n) > \eta$ for all sufficiently large $n$. Let $0<\e<\eta/4$ be arbitrary.
  Since $\diam(I_n) \to 0$,  we have $\diam(I_n) < \e$ for all large $n$. Noting that $C_n - I_n$ is a connected subset of $\Omega$, 
  let $\Omega_n$ be the connected component of $\Omega$ containing
  $C_n - I_n$. Then $C_n$ is contained in the $\e$-neighborhood of $\Omega_n$, which implies $$\diam(\Omega_n) \ge \diam (C_n) - 2 \e  > \eta/2.$$
  By Corollary \ref{diam}, the collection $\{\Omega_n: \diam(\Omega_n) > \eta/2\}$ must be a finite set, say, $\{\Omega_1, \cdots, \Omega_N\}$.
  Therefore, for some $1\le j\le N$, there exists an
  infinite sequence $C_{n_i}$  contained in the $\e$-neighborhood of $\Omega_j$.
  Hence $C$ is contained in the $2\e$-neighborhood of $ \Omega_j$. Since the collection $\{\Omega_1, \cdots, \Omega_N\}$ does not depend on $\e$,
  we can find a sequence $\e_k\to 0$ and a fixed $1\le j\le N$ such that
   $C$ is contained in the $2\e_k$-neighborhood of $\Omega_j$.
  It follows  that $C \subset \overline{\Omega_j}$, contradicting the hypothesis on $C$. This finishes the proof.

\section{Self-joinings of Kleinian groups and Proof of Theorem \ref{m1}.}
Let $\Ga<\PSL_2(\bc)$ be a Zariski dense discrete subgroup with limit set $\La$. As before, we denote by $\Omega=\S^2-\La$ its ordinary set.

We fix a discrete faithful representation $\rho:\Ga\to \PSL_2(\bc)$ such that $\rho(\Ga)$ is Zariski dense.

We now define the self-joining of $\Ga$ via $\rho$ as  
$$\Ga_\rho:=(\id\times \rho)(\Ga) = \{ (\ga, \rho(\ga)) : \ga \in \Ga \},$$
which is a discrete subgroup of $G$.

We begin by recalling two standard facts:
\begin{lem} \label{Zdense}
The subgroup $\Ga_\rho$ is Zariski dense in $G$ if and only if $\rho$ is not a conjugation  by an element of $\op{M\ddot{o}b}(\S^2)$.
\end{lem}
\begin{proof}

It is clear that if $\rho$ is a conjugation by an element of $\Mob(\S^2)$, then $\Gr$ is not Zariski dense in $G$. To see the converse, let $G_0<G$ be the Zariski closure of $\Gr$ and suppose that $G_0 \neq G$. Denote by $\pi_i : G = \PSL_2(\bc) \times \PSL_2(\bc) \to \PSL_2(\bc)$ the projection onto the $i$-th component.

We now claim that $\pi_1|_{G_0}$ is an isomorphism. Since $\Gamma$ is Zariski dense, $\pi_1|_{G_0}$ is surjective. Hence, it suffices to show that $\pi_1|_{G_0}$ is injective. Note that $G_0 \cap \ker \pi_1 = G_0 \cap ( \{e \} \times \PSL_2(\bc))$ is a normal subgroup of $G_0$. Hence, $G_0 \cap \ker \pi_1$ is normalized by  $\{e\} \times \PSL_2(\bc)$ since $\rho(\Gamma)$ is Zariski dense $\PSL_2(\bc)$. Thus, $G_0 \cap \ker \pi_1$ is a normal subgroup of $\ker \pi_1$. As $\ker \pi_1 \cong \PSL_2(\bc)$ is simple, $G_0 \cap \ker \pi_1$ is either trivial or $\{e \} \times \PSL_2(\bc)$. In the latter case, note that $\{e \} \times \PSL_2(\bc)  < G_0$.  
Since $\pi_1|_{G_0}$ is surjective,
it follows that $G_0 = G$, yielding contradiction. Therefore $\pi_1|_{G_0}$ is injective, and hence an isomorphism.
Similarly, $\pi_2|_{G_0}$ is an isomorphism. Hence,  $\pi_2|_{G_0} \circ \pi_1|_{G_0}^{-1}$ is a Lie group automorphism of $\PSL_2(\bc)$. 
Hence it is a conjugation by a M\"obius transformation (cf. \cite{Gundogan2010component}). Since this map restricts to $\rho$ on $\Gamma$, it finishes the proof.
\end{proof}

Since $\rho$ gives an isomorphism from $\Ga$ to $\rho(\Ga)$
and $f$ is an equivariant embedding, it follows that
$\rho$ maps every loxodromic element $\ga$ to a loxodromic element $\rho(\ga)$
and $f$ sends the attracting fixed point of  $\ga\in \Ga$
to the attracting fixed point of $\rho(\gamma)$. Since the set of attracting fixed points of loxodromic elements of $\Ga$ is dense in $\La$, this implies the following.
\begin{lem}\label{boundary}
 There can be at most one $\rho$-boundary map $ f :~\La \to \S^2$.
In particular, if $\rho$ is a conjugation by $g \in \op{M\ddot{o}b}(\S^2)$, then $f = g|_{\La}$.
\end{lem}

\subsection*{Proof of Theorem \ref{m1}}
By Lemma \ref{stable}, Theorem \ref{m1} follows from the following:
\begin{theorem}\label{m3} 
Let $\Ga<\PSL_2(\c)$ be a Zariski dense Kleinian group such that $\La$ is doubly stable.
Let $\rho\in \dGa$ be a Zariski dense representation with boundary map $f:\La\to \S^2$. 
 Unless $\rho$ is a conjugation, 
the subset
\be\label{cccc} \La_f:=\bigcup \{C\cap \La: f(C\cap \La)\text{ is contained in a circle}\}\ee
has empty interior in $\La$; hence
$$\{C\in \cal C_\La: f(C\cap \La)\text{ is contained in a circle}\}$$
has empty interior in $\cal C_\La$.
\end{theorem}
\begin{proof}
Suppose that $\rho$ is not a conjugation, so that
 $\Gr$ is Zariski dense by Lemma \ref{Zdense}.
It  follows easily from the minimality of  the limit set $\La_\rho$ of $\Ga_\rho$ that 
 \be\label{ff} \La_\rho=\{(\xi, f(\xi))\in \S^2 \times \S^2 : \xi \in \La \}.\ee

Let $\tilde \La_{\Gamma_\rho}$ be as in Theorem \ref{densec}, which must be 
 of the form $\{(\xi, f(\xi)): \xi \in \tilde \La\}$
 for some dense subset $\tilde \La$ of $ \La$.
 
  Suppose on the contrary that $\La_f$ has non-empty interior. Then $\La_f\cap 
 \tilde \La\ne \emptyset$. 
 It follows that there exists $C_0\in \cal C_\La$ such that $C_0\cap \tilde \La \ne \emptyset$ and $f(C_0\cap \La)$ is contained in some circle, say, $D_0$.
 Set $T_0=(C_0, D_0)$.
 Since $C_0\cap \tilde \La\ne\emptyset$, it follows from Theorem \ref{densec} that
 \be\label{nd3} \overline{\Ga_\rho T_0}= \cal T_\rho\ee 
where $\cal T_\rho=\cal T_{ {\Ga_\rho}}$ is the space of all tori intersecting $\La_\rho$.
 On the other hand, we now show 
 that the condition $f(C_0\cap \La)\subset D_0$
 implies that
 $\Ga_\rho T_0$ cannot be dense in $\cal T_{\rho}$,
 using Lemma \ref{limsuplim}.
 
 \medskip
 
 \noindent{\bf Step 1:} 
 There exists a circle $D$ which intersects $\La_{\rho(\Ga)}$ precisely at one point, say $f(\xi_0)$.
To show this, fix any $f(\xi)\in \La_{\rho(\Ga)}$ and let $D'$ be the boundary of the minimal disk $B'$ centered at $f(\xi)$ which contains all of $\La_{\rho(\Ga)}$. By the minimality of $B'$,
$D'\cap \Lambda_{\rho(\Ga)}\ne \emptyset$.
Choose $f(\xi_0)\in D'\cap \La_{\rho(\Ga)}$, and let $D$ be a circle tangent to $D'$ at $f(\xi_0)$ which does not intersect the interior of $B'$.
 
 \medskip
 
 \noindent{\bf Step 2:} By the hypothesis that $\La$ is doubly stable, 
we can  find a circle $C$ containing $\xi_0$ which is doubly stable for $\La$.
 
 \medskip

 \noindent{\bf Step 3:}  Setting $T=(C, D)$, we have $T\notin \overline{\Ga_\rho T_1}$ for any torus $T_1=(C_1, D_1)$ with $f(C_1\cap \La)\subset D_1$. In particular, $T\notin \overline{\Ga_\rho T_0}$.
 
Suppose on the contrary that  there exists a sequence $\ga_n\in \Ga$ such that
  $\ga_n C_1$ converges to $C$ and $\rho(\ga_n) D_1$ converges to $D$.
  Since $C$ is doubly stable for $\La$, we have
  \be \label{at2} \# \limsup (\ga_n C_1 \cap \La) \ge 2.\ee

By the $\rho$-equivariance of $f$, we have
$$  f( \ga_n C_1 \cap \La ) =f(\ga_n (C_1\cap \La))=
\rho(\ga_n) f(C_1\cap \La) \subset \rho(\ga_n) D_1 \cap \La_{\rho(\Ga)} .$$ 
Hence
$$\limsup f( \ga_n C_1 \cap \La ) \subset \limsup 
  (\rho(\ga_n) D_1 \cap \La_{\rho(\Ga)}) \subset D\cap \La_{\rho(\Ga)}.$$
 It now follows from \eqref{at2} and the injectivity of $f$ that
  $$\# D\cap \La_{\rho(\Ga)}\ge 2.$$
  This contradicts the choice of $D$ made in Step (1), hence proving Step (3).
 
 \medskip
 
  Since $(\xi_0, f(\xi_0))\in T\cap \La_\rho$, we have
  $T\in \cal T_{\rho}$.
Hence we obtained a contradiction to \eqref{nd3}. Therefore $\La_f$ has empty interior, completing the proof.
 \end{proof}

\bibliographystyle{plain}

\end{document}